\documentclass{amsart}

\usepackage{graphicx}
\usepackage{pinlabel}

\newtheorem{theorem}{Theorem}
\newtheorem{prop}[theorem]{Proposition}
\newtheorem{lemma}[theorem]{Lemma}
\newtheorem{corollary}[theorem]{Corollary}
\newtheorem{definition}[theorem]{Definition}

\def\T{\mathcal{T}}
\def\R{\mathbb{R}}

\begin{document}

\title[discrete uniformization hyperbolic]{A discrete uniformization theorem
for polyhedral surfaces II}

\author{Xianfeng Gu}

\address{Department of Computer Science, Stony Brook University, Stony Brook, NY, 11794}

\email{gu@cs.stonybrook.edu}

\author{Ren Guo}

\address{Department of Mathematics, Oregon State University, Corvallis, OR, 97330}

\email{guore@math.oregonstate.edu}

\author{Feng Luo}

\address{Department of Mathematics, Rutgers University, Piscataway, NJ, 08854}

\email{fluo@math.rutgers.edu}

\author{Jian Sun}

\address{Mathematical Science Center, Tsinghua University, Beijing, 100084, China}

\email{jsun@math.tsinghua.edu.cn}

\author{Tianqi Wu}

\address{Mathematical Science Center, Tsinghua University, Beijing, 100084, China}

\email{mike890505@gmail.com}



\begin{abstract} A discrete conformality for
hyperbolic polyhedral surfaces is introduced in this paper. This
discrete conformality is shown to be computable.
 It is proved that each hyperbolic polyhedral metric on a closed
  surface is discrete conformal to a unique hyperbolic polyhedral
   metric with a given discrete curvature satisfying Gauss-Bonnet formula. Furthermore, the hyperbolic
   polyhedral metric with given curvature can be obtained using a
   discrete Yamabe flow with surgery. In particular, each hyperbolic
   polyhedral metric on a closed surface with negative Euler characteristic is
   discrete conformal to a unique hyperbolic metric.
\end{abstract}

\maketitle

\section{Introduction}
\subsection{Statement of results}

This is a continuation of  \cite{glsw} in which a discrete
uniformization theorem for Euclidean polyhedral metrics on
closed surfaces is established. 
The purpose of this paper is to prove the counterpart of discrete
uniformization for hyperbolic polyhedral metrics. In particular,
we will introduce a discrete conformality for hyperbolic
polyhedral metrics on surfaces and show the discrete conformality
is algorithmic.

Recall that a \it marked surface \rm $(S,V)$ is a pair of a closed
connected surface $S$ together with a finite non-empty subset $V
$. A \it triangulation \rm of a marked surface $(S,V)$ is a
triangulation of $S$ so that its vertex set is $V$.
 A \it
hyperbolic polyhedral metric $d$ \rm on a marked surface $(S, V)$
is obtained as the isometric gluing of hyperbolic triangles along
pairs of edges so that its cone points are in $V$. It is the same
as a hyperbolic cone metric on $S$ with cone points in $V$. We use
the terminology \it polyhedral metrics \rm to emphasize that these
metrics are determined by finite sets of data (i.e., the finite
set of lengths of edges).
 Every hyperbolic
polyhedral metric has an associated \it Delaunay triangulation \rm
which has the property that the interior of the circumcircle of
each triangle contains no other vertices.

\begin{definition}\label{12345}(discrete conformality)
Two hyperbolic polyhedral metrics $d, d'$ on a closed marked
surface $(S, V)$ are {\it discrete conformal} if there exist a
sequence of hyperbolic polyhedral metrics $d_1=d, d_2, ...,
d_m=d'$ on $(S, V)$ and triangulations $\mathcal{T}_1,
\mathcal{T}_2, ..., \mathcal{T}_m$ of $(S, V)$ satisfying
\begin{itemize}
\item[(a)] each $\mathcal{T}_i$ is Delaunay in $d_i$,

\item[(b)] if $\mathcal{T}_i=\mathcal{T}_{i+1}$, there exists a function $u: V\to \mathbb{R}$,
 called a {\it conformal factor}, so that if $e$ is an edge in $\mathcal{T}_i$ with end
 points $v$ and $v'$, then the lengths $x_{d_i}(e)$ and $x_{d_{i+1}}(e)$ of $e$ in metrics $d_i$ and $d_{i+1}$ are related by
$$\sinh \frac{x_{d_{i+1}}(e)}2= e^{u(v)+u(v')} \sinh \frac{x_{d_i}(e)}2,$$

\item[(c)] if $\mathcal{T}_i\neq \mathcal{T}_{i+1},$ then $(S, d_i)$ is isometric to
$(S, d_{i+1})$ by an isometry homotopic to the identity in $(S, V)$.
\end{itemize}
\end{definition}

This definition is the hyperbolic counterpart of discrete
conformality introduced in \cite{glsw}. The condition (b) first
appeared in
\cite{bps}. 


\begin{theorem}\label{algo} Suppose $d$ and $d'$ are two hyperbolic (or Euclidean) polyhedral metrics
given as isometric gluings of geometric triangles on a closed
marked surface $(S,V)$. There exists an algorithm to decide if $d$
and $d'$ are discrete conformal.
\end{theorem}

The above theorem shows that discrete conformality is computable.
This is in contrasts to the conformality in Riemannian geometry.
Indeed, it is highly unlikely that there exist algorithms to
decide if two hyperbolic (or Euclidean) polyhedral metrics on
$(S,V)$ are conformal in the Riemannian sense.

 The \it discrete curvature $K$ \rm of a polyhedral metric $d$ is
the function defined on $V$ sending $v \in V$ to $2\pi$ less cone
angle at $v$.  It is well known that the discrete curvature
satisfies the Gauss-Bonnet identity $\sum_{v \in V}K(v)
=2\pi\chi(S) +Area(d)$ where $ Area(d)$ is the area of the metric
$d.$

\begin{theorem}\label{thm:hyperbolic}
 Suppose $(S,V)$ is a closed connected marked surface and $d$
 is a hyperbolic polyhedral metric on $(S, V)$. Then for any $K^*: V \to (-\infty, 2\pi)$ with \\
$\sum_{v\in V} K^*(v)> 2\pi \chi(S),$ there exists a unique
hyperbolic polyhedral metric $d'$ on $(S, V)$ so that $d'$ is
discrete conformal to $d$ and the discrete curvature of $d'$ is
$K^*$. Furthermore, the discrete Yamabe flow with surgery
associated to curvature $K^*$ having initial value $d$ converges
to $d'$ exponentially fast.
\end{theorem}

In particular, on a closed connected surface $S$ with $\chi(S)<0$,
by choosing $K^*=0$, we obtain,

\begin{corollary}(discrete uniformization) Let $S$ be a closed
connected surface of negative Euler characteristic and $V \subset
S$ be a finite non-empty subset. Then each hyperbolic polyhedral
metric $d$ on $(S,V)$ is discrete conformal to a unique hyperbolic
metric $d^*$ on the surface $S$. Furthermore, there exists a
$C^1$-smooth flow on the Teichmuller space of hyperbolic
polyhedral metrics on $(S,V)$ which preserves discrete conformal
classes and flows each polyhedral metric $d$ to $d^*$ as time goes
to infinity.
\end{corollary}

\subsection{Basic idea of the proof}
The basic idea of the proof is similar to that of \cite{glsw}. We
first introduce the Teichm\"uller space $T_{hp}(S,V)$ of
hyperbolic polyhedral metrics on $(S,V)$. It is shown to be a real
analytic manifold which admits  a cell decomposition by the work
of \cite{Leibon} and \cite{ha}. Using the work of Kubota \cite{k}
on hyperbolic Ptolemy identity and the work of Penner
\cite{penner}, we show that $T_{hp}(S,V)$ is $C^1$ diffeomorphic
to the decorated Teichm\"uller space so that two hyperbolic
polyhedral metrics are discrete conformal if and only if their
corresponding decorated metrics have the same underlying
hyperbolic structure. Using this correspondence, we show Theorem
\ref{thm:hyperbolic} using a variational principle first appeared
in \cite{bps}.

Many arguments in this paper are similar to that of \cite{glsw}.
 The
major difference between Euclidean and hyperbolic polyhedral
metrics comes from the circumcircles of triangles. Namely, the
circumcircle of a hyperbolic triangle may be non-compact, i.e., a
horocyle or a curve of constant distance to a geodesic. This
creates many difficulties when one uses the inner angle
characterization of Delaunay triangulations. To overcome this, we
prove (theorem \ref{9}) that every triangle in a Delaunay
triangulation of a hyperbolic polyhedral metric on a compact
surface has compact circumcircle.

\subsection{Organization of the paper}
Section 2 deals with the Teichm\"uller space of hyperbolic
polyhedral metrics, its analytic cell decomposition and Delaunay
triangulations. In section 3, we show that there is a $C^1$
diffeomorphism between the Teichm\"uller space of hyperbolic
polyhedral metrics and  the decorated Teichm\"uller space.
Section 4 is devoted to the proof
 of Theorem \ref{thm:hyperbolic}. Section 5 proves theorem \ref{algo}.  In the appendix, a
 technical lemma is proved.

\subsection{Acknowledgement} The work is supported in part by the NSF of USA and
the NSF of China.

\section{Teichm\"uller space of polyhedral metrics}

\subsection{Triangulations and some conventions}
Take a finite disjoint union $X$ of triangles and identify edges
in pairs by homeomorphisms. The quotient space $S$ is a compact
surface together with a \it triangulation \rm $\T$ whose simplices
are the quotients of the simplices in the disjoint union $X$. Let
$V=V(\T)$ and $E=E(\T)$ be the sets of vertices and edges in $\T$.
We call $\T$ a \it triangulation \rm of the marked surface
$(S,V)$.  If each triangle in the disjoint union $X$ is hyperbolic
and the identification maps are isometries, then the quotient
metric $d$ on the quotient space $(S,V)$ is a called \it
hyperbolic polyhedral metric\rm. The set of cone points of $d$ is
in $V$. Given a hyperbolic polyhedral metric $d$ and a
triangulation $\T$ on $(S, V)$, if each triangle in $\T$ (in
metric $d$) is isometric to a hyperbolic triangle, we say $\T$ is
\it geodesic \rm in $d$. If $\T$ is a triangulation of $(S,V)$
isotopic to a geometric triangulation $\T'$ in a hyperbolic
polyhedral metric $d$, then the \it length \rm of an edge  $e \in
E(\T)$ (or \it angle \rm of a triangle at a vertex in $\T$) is
defined to be the length (respectively angle) of the corresponding
geodesic edge $e' \in E (\T')$ (triangle at the vertex) measured
in metric $d$.

Suppose $e$ is an edge in $\T$ adjacent to two distinct triangles
$t, t'$. Then the \it diagonal switch \rm on $\T$ is a new
triangulation $\T'$ obtained from $\T$ by  replaces $e$ by the
other diagonal in the quadrilateral $t \cup_e t'$.

For simplicity, the terms metrics and triangulations in many
places will mean isotopy classes of metrics and isotopy classes of
triangulations. They can be understood from the context without
causing confusion.

 If $X$ is a finite
set, $|X|$ denotes its cardinality and $\R^X$ denotes the vector
space $\{f: X \to \R\}$. For a finite vertex set $W=\{w_1, ...,
w_m\}$, we identify $\R^W$ with $\R^m$ by sending $x \in \R^m$ to
$(x(w_1), ..., x(w_m))$.

All surfaces are assumed to be compact and connected in the rest
of the paper.

\subsection{The Teichm\"uller space and the length coordinates}

Two hyperbolic polyhedral metrics $d, d'$ on $(S, V)$ are called
{\it Teichm\"uller equivalent} if there is an isometry $h: (S, V,
d)\to (S, V, d')$ so that $h$ is isotopic to the identity map on
$(S, V)$. The {\it Teichm\"uller space} of all hyperbolic
polyhedral metrics on $(S, V)$, denoted by $T_{hp}(S, V)$, 
is the set of all Teichm\"uller equivalence classes of hyperbolic polyhedral
metrics on $(S, V)$.

\begin{lemma}
$T_{hp}(S, V)$ is a real analytic manifold.
\end{lemma}

\begin{proof}
Suppose $\mathcal{T}$ is a triangulation of $(S, V)$ with the set
of edges $E=E(\mathcal{T})$. Let

\begin{align*}
& \R^{E(\T)}_{\Delta} =\{ x \in \R_{>0}^E | \text{$\forall$
triangle $t$ in $\T$ with edges $e_i, e_j, e_k$,} \quad
 x(e_i)+x(e_j) > x(e_k)\}
\end{align*}
be the convex polytope in $\R^E$.  For each $x \in
\R^{E(\T)}_{\Delta}$, one constructs a hyperbolic polyhedral
metric $d_x$ on $(S, V)$ by replacing each triangle $t$ of edges
$e_i, e_j, e_k$ by a hyperbolic triangle of edge lengths $x(e_i),
x(e_j), x(e_k)$ and gluing them by isometries along the
corresponding edges. This construction produces an injective map
(the length coordinate associated to $\T$)
$$ \Phi_{\T}: \R^{E(\T)}_{\Delta} \to T_{hp}(S, V)$$ sending $x$ to  $[d_x]$.
The image $P(\T) :=\Phi_{\T}(\R_{\Delta}^{E(\T)})$ is the space of
all hyperbolic polyhedral metrics $[d]$ on $(S, V)$ for which $\T$
is isotopic to a geodesic triangulation in $d$. We call $x$  the
\it length coordinate \rm of $d_x$ and $[d_x]=\Phi_{\T}(x)$ (with
respect to $\T$). In general $P(\T) \neq T_{hp}(S, V)$ (see \S2.1
in \cite{glsw}).

Since each hyperbolic polyhedral metric on $(S,V)$ admits a
geometric triangulation (for instance its Delaunay triangulation),
we see that $T_{hp}(S, V) =\cup_{\T} P(\T)$ where the union is
over all triangulations of $(S,V)$. The space $T_{hp}(S, V)$ is a
real analytic manifold with real analytic coordinate charts $\{
(P(\T), \Phi_{\T}^{-1}) | \T$ triangulations of $(S,V)$\}. To see
transition functions $\Phi_{\T}^{-1}\Phi_{\T'}$ are real analytic,
note that  any two triangulations of $(S,V)$ are related by a
sequence of diagonal switches. Therefore, it suffices to show the
result for $\T$ and $\T'$ which are related  by a diagonal switch
along an edge $e$. In this case, the transition function
$\Phi_{\T}^{-1}\Phi_{\T'}$ sends $(x_0,x_1, ...., x_m)$ to
$(f(x_0, ..., x_m), x_1, ..., x_m)$ where $x_0$ is the length of
$e$ and $f$ is the length of the diagonal switched edge. Let $t,
t'$ be the triangles adjacent to $e$ so that the lengths of edges
of $t, t'$ are \{$x_0, x_1, x_2$\} and \{$x_0, x_3, x_4$\}. Using
the cosine law, we see that $f$ is a real analytic function of
$x_0, ..., x_4$.
\end{proof}

\subsection{Delaunay triangulations and marked quadrilaterals}

Each hyperbolic triangle $t$ in $\mathbf H^2$ has a circumcircle
which is the curve of constant geodesic curvature containing the
three vertices of $t$. When the circumcircle is compact, it is a
hyperbolic circle. When it is not compact, it is either a
horocycle or a curve of constant distance to a geodesic.  We call
the convex region bounded by the circumcircle the \it circum-ball
\rm of the triangle $t$.  A \it marked quadrilateral \rm $Q$ is a
hyperbolic quadrilateral together with a diagonal $e$ inside $Q$.
It is the same as a union of two hyperbolic triangles $t, t'$
along a common edge $e$, i.e., $Q =t \cup_e t'$. A hyperbolic
polygon  is called \it cyclic \rm if its vertices lie in a curve
of constant geodesic curvature in the hyperbolic plane. A marked
quadrilateral $t \cup_e t'$ is cyclic if and only if the two
circumecircles for $t$ and $t'$ coincide.

 A geodesic triangulation $\T$ of a
hyperbolic polyhedral surface $(S,V,d)$ is said to be \it Delaunay
\rm if for each edge $e$ adjacent to two hyperbolic triangles $t$
and $t'$, the interior of the circumball of $t$ does not contain
the vertices of $t'$ when the quadrilateral $t \cup_e t'$ is
lifted to $\mathbf H^2$. The last condition is sometimes called
the \it empty ball condition\rm.  We will call the marked
quadrilateral $t\cup_e  t'$ the \it quadrilateral associated to
the edge \rm $e$.  G. Leibon \cite{Leibon} gave a very nice
algebraic description of empty-ball condition in terms of the
inner angles.

\begin{lemma}[Leibon]\label{thm:Leibon} A geodesic triangulation $\T$ is
 Delaunay  if
and only if
\begin{equation}\label{eq:leibon} \alpha+\alpha'\leq \beta+\beta'+\gamma+\gamma' \end{equation}
for each edge $e$, where $\alpha, \beta, \gamma, \alpha', \beta',
\gamma'$ are angles of the two triangles in $\T$ having $e$ as the
 common edge so that $\alpha$ and $\alpha'$ are opposite to $e$.
 Furthermore, the equality holds for $e$ if and only if the marked
 quadrilateral associated to $e$ is cyclic.
\end{lemma}

The inequality (\ref{eq:leibon}) can be expressed in terms of the
edge lengths as follows.

\begin{prop}\label{thm:delaunay condition}
A geodesic triangulation $\T$ is  Delaunay  if and only if
\begin{equation}\label{eq:algdel}
\frac{\sinh^2(x_1/2)+\sinh^2(x_2/2)-\sinh^2(x_0/2)}{\sinh(x_1/2)\sinh(x_2/2)}
+\frac{\sinh^2(x_3/2)+\sinh^2(x_4/2)-\sinh^2(x_0/2)}{\sinh(x_3/2)\sinh(x_4/2)}\geq
0
\end{equation}
for each edge $e$ adjacent two triangles $t,t'$ of edge lengths
$x_0,x_1,x_2$ and $x_0,x_3,x_4$ respectively. Furthermore, the
equality holds for an edge $e$ if and only if $t\cup_e t'$ is
cyclic.
\end{prop}

\begin{proof}
We begin with

\begin{lemma}\label{thm:relation} Let $x_1,x_2,x_3$ be side lengths of a hyperbolic triangle and
$a_1,a_2,a_3$ be the opposite angles so that $a_i$ is facing the
edge of length $x_i$. Then
$$2\sin\frac{a_2+a_3-a_1}2 \cdot \cosh\frac{x_1}2= \frac{\sinh^2(x_2/2)+\sinh^2(x_3/2)-\sinh^2(x_1/2)}
{\sinh(x_2/2)\sinh(x_3/2)}.$$
\end{lemma}

\begin{proof} By the cosine law expressing $x_i$ in terms of $a_1,
a_2, a_3$, we have
\begin{align*}
&\sinh^2(x_2/2)+\sinh^2(x_3/2)-\sinh^2(x_1/2)\\
&=\frac12( \cosh(x_2)+\cosh(x_3)-\cosh(x_1)-1 )\\
&=\frac12 [ \frac{\cos a_2+\cos a_1 \cos a_3}{\sin a_1 \sin a_3}+
\frac{\cos a_3+\cos a_1 \cos a_2}{\sin a_1 \sin a_2} -
                      \frac{\cos a_1+\cos a_2 \cos a_3}{\sin a_2 \sin a_3} - 1 ] \\
&=\frac1{2\sin a_1 \sin a_2 \sin a_3}  (\sin(a_2+a_3)-\sin a_1)(\cos a_1 +\cos (a_2-a_3))\\
&= \frac{2\sin\frac{a_2+a_3-a_1}2 \cos\frac{a_1+a_2+a_3}2
\cos\frac{a_1+a_2-a_3}2 \cos\frac{a_1-a_2+a_3}2}{\sin a_1 \sin a_2
\sin a_3}.
\end{align*}

On the other hand,
\begin{align*}
\sinh^2(x_i/2)&=\frac12( \cosh x_i -1) \\
&=\frac12 (\frac{\cos a_i+\cos a_j \cos a_k}{\sin a_j \sin a_k}-1) \\
&=\frac12 \frac{\cos a_i+ \cos(a_j+a_k)}{\sin a_j \sin a_k}\\
&=\frac{\cos\frac{a_i+a_j+a_k}2 \cos\frac{a_i-a_j-a_k}2}{\sin a_j \sin a_k}.
\end{align*}

Therefore
\begin{align*}
&\frac{\sinh^2(x_2/2)+\sinh^2(x_3/2)-\sinh^2(x_1/2)}{\sinh(x_2/2)\sinh(x_3/2)}\\
&= \frac{2\sin\frac{a_2+a_3-a_1}2 \cos\frac{a_1+a_2+a_3}2 \cos\frac{a_1+a_2-a_3}2 \cos\frac{a_1-a_2+a_3}2}
{\sin a_1 \sin a_2 \sin a_3 \sqrt{\frac{\cos\frac{a_1+a_2+a_3}2 \cos\frac{a_2-a_1-a_3}2}{\sin a_1 \sin a_3}}
                            \sqrt{\frac{\cos\frac{a_1+a_2+a_3}2 \cos\frac{a_3-a_1-a_2}2}{\sin a_1 \sin a_2}}} \\
&= 2\sin\frac{a_2+a_3-a_1}2 \cdot \sqrt{\frac{\cos\frac{a_1+a_2-a_3}2 \cos\frac{a_1-a_2+a_3}2}{\sin a_2 \sin a_3}}     \\
&= 2\sin\frac{a_2+a_3-a_1}2 \cdot \cosh\frac{x_1}2.
\end{align*}

In the last step above,  we have used \begin{align*}
(\cosh\frac{x_1}2)^2&=\frac12( \cosh x_1 +1) \\
&=\frac12 (\frac{\cos a_1+\cos a_2 \cos a_3}{\sin a_2 \sin a_3}+1) \\
&=\frac12 \frac{\cos a_1+ \cos(a_2-a_3)}{\sin a_2 \sin a_3}\\
&=\frac{\cos\frac{a_1+a_2-a_3}2 \cos\frac{a_1-a_2+a_3}2}{\sin a_2 \sin a_3}.
\end{align*}

\end{proof}


Now (\ref{eq:leibon}) is equivalent to
$\sin\frac{\beta+\gamma-\alpha}2+\sin\frac{\beta'+\gamma'-\alpha'}2\geq
0.$ By Lemma \ref{thm:relation} applied to triangles of lengths
$\{x_0, x_1, x_2\}$ and $\{x_0, x_3, x_4\}$, we see that Delaunay
is equivalent to (\ref{eq:algdel}).
\end{proof}

\subsection{Delaunay triangulations of compact hyperbolic polyhedral surfaces}


\begin{theorem}\label{9} If $\T$ is a
Delaunay triangulation of a closed hyperbolic polyhedral surface
$(S,V,d)$, then each triangle has a compact circumcircle.
\end{theorem}

\begin{proof}
By Proposition \ref{thm:delaunay condition} for Delaunay
triangulations inequality (\ref{eq:algdel}) holds. On the other
hand, by lemma 4.2 of \cite{glsw},

\begin{lemma}[\cite{glsw}] \label{thm:deltri}
Suppose $y: E(\T) \to \R_{>0}$ is a function satisfying for each
edge $e_0$ adjacent to two triangles $t,t'$ of edges $e_0,e_1,
e_2$ and $e_0, e_3, e_4$
$$ \frac{ y_1^2+y_2^2-y_0^2}{y_1y_2} +
\frac{y_3^2+y_4^2-y_0^2}{y_3y_4} \geq 0$$ where $y_i =y(e_i)$.
Then $y(e_i) + y(e_j) > y(e_k)$
whenever $e_i, e_j, e_k$ form edges of a triangle in $\T$.
\end{lemma}

Taking $y(e) =\sinh(\frac{x(e)}2)$ in the above lemma and using
(\ref{eq:algdel}), we obtain
\begin{equation}\label{trig}\sinh(\frac{x(e_i)}2) +\sinh(\frac{x(e_j)}2) > \sinh(\frac{x(e_k)}2). \end{equation}

Now  theorem \ref{9} follows from (\ref{trig}) and a result in
\cite{Fenchel} page 118,

\begin{prop}[Fenchel]\label{thm:cycle}
Let $C$ be the circumcircle of a hyperbolic triangle of edge
lengths $x_i, x_j, x_k$.  Then
 $C$ is a (compact) hyperbolic circle  if and only if
$\sinh(\frac{x_i}2)+\sinh(\frac{x_j}2) > \sinh(\frac{x_k}2)$.



\end{prop} \end{proof}
Since $\sinh(a+b) > \sinh(a)+\sinh(b)$ for $a,b>0$, by
(\ref{trig}), we obtain
\begin{equation}\label{trieq} x(e_i)+x(e_j)> x(e_k), \end{equation}
whenever $e_i, e_j, e_k$ form edges of a triangle. This implies,
\begin{corollary}\label{12} Suppose $x: E(\T) \to \R_{>0}$ is a function so
that (\ref{eq:algdel}) holds at each edge. Then $x$ is the edge
length function (in $\T$) of a hyperbolic polyhedral metric on
$(S, V)$.

\end{corollary}


It is highly likely that theorem \ref{9} still holds for
hyperbolic cone metrics on high dimensional compact manifolds,
i.e., empty-ball condition implies compact circumsphere. The work
of \cite{De} shows that it holds for decorated finite volume
hyperbolic metrics of any dimension. 

The classical way of constructing many Delaunay triangulations of
a polyhedral metric
 $d$ on $(S, V)$ is as follows. See for instance \cite{bs}. Define the {\it
Voronoi decomposition} of $(S,V,d)$ to be the collection of
2-cells $\{R(v)| v\in V\}$ where $R(v)=\{x\in S | d(x,v)\leq
d(x,v')$ for all $v'\in V\}$. Its dual is called a {\it Delaunay
tessellation} $\mathcal{C}(d)$ of $(S,V,d)$. It is a cell
decomposition of $(S,V)$ with vertices $V$ and two vertices $v,
v'$ jointed by an edge if and only if $R(v) \cap R(v')$ is
1-dimensional. By definition, each 2-cell in the Delaunay
tessellation is a convex polygon inscribed to a compact circle in
$\mathbf H^2$ whose center is a vertex of the Voronoi
decomposition.
 By further triangulating all non-triangular
2-dimensional cells (without introducing extra vertices) in
$\mathcal{C}(d)$, one obtains a Delaunay triangulation of
$(S,V,d)$.  This Delaunay triangulation has the property that the
circumcircles of triangles are hyperbolic circles (i.e., compact).
Indeed, the centers of the circumcircles are the vertices in the
Voronoi cell decomposition.  Conversely, if $\T$ is a Delaunay
triangulation with compact circumcircles for all triangles, then
it is a triangulation of the Delaunay tessellation.
Combining theorem \ref{9}, we obtain
part (a) of the following,

\begin{prop} \label{134} (a) Suppose  $\T$ is a geodesic triangulation of a
compact hyperbolic polyhedral surface $(S,V,d)$. Then $\T$
satisfies the empty-ball condition if and only if it is a geodesic
triangulation of the Delaunay tessellation.

(b) If $\mathcal{T}$ and $\mathcal{T}'$ are Delaunay
triangulations of a hyperbolic polyhedral metric $d$ on a closed
marked surface $(S,V)$, then there exists a sequence of Delaunay
triangulations $\mathcal{T}_1=\mathcal{T}, \mathcal{T}_2,...,
\mathcal{T}_k=\mathcal{T}'$ of $d$ so that $\mathcal{T}_{i+1}$ is
obtained from $\mathcal{T}_i$ by a diagonal switch.

(c) Suppose $\T$ is a Delaunay triangulation of a compact
hyperbolic polyhedral surface $(S,V,d)$ whose diameter is $D$.
Then the length of each edge $e$ in $\T$ is at most $2D$.  In
particular, there exists an algorithm to find all Delaunay
triangulations of a hyperbolic polyhedral surface.
\end{prop}
\begin{proof}
Part(b) of the proposition follows from part(a) and the well known
fact that any two geodesic triangulations of the Delaunay
tessellation are related by a sequence of diagonal switches.
Indeed, any two geodesic triangulations of a convex cyclic polygon
are related by a sequence of (geodesic) diagonal switches. See for
instance \cite{bs} for a proof.

To see part (c), if $e$ is an edge dual to two Voronoi cells
$R(v)$ and $R(v')$, then the length of $e$ is at most the sum of
the diameters of $R(v)$ and $R(v')$. However, the diameters of
$R(v)$ and $R(v')$ are bounded by the diameter of the surface $S$.
Thus, the length of $e$ is at most $2D$.   It is well known that
for any constant $C$, there exists an algorithm to list all
geodesic paths in $(S,V,d)$ of lengths at most $C$ joining $V$ to
$V$. Therefore, we can list algorithmically all Delaunay
triangulations of a given polyhedral metric on $(S,V)$.

\end{proof}
Note that if we remove the compactness of the space $S$, then
there are examples of geodesic triangulations with empty-ball
condition which does not come from dual of Voronoi cell. See
\cite{De}.

For a triangulation $\mathcal{T}$ of $(S,V)$, the associated
Delaunay cell in $T_{hp}(S, V)$ is defined to be
\begin{align*}
&D_c(\mathcal{T}) &=\{[d]\in T_{hp}(S, V) | \mathcal{T}\ \mbox{is
isotopic to a Delaunay triangulation of}\ d\}.
\end{align*}

Theorem \ref{9} and corollary \ref{12} show that
$D_c(\mathcal{T})$ is defined by a finite set of real analytic
inequalities (i.e., (\ref{eq:algdel})). On the other hand, Leibon
showed in \cite{Leibon} that $D_c(\mathcal{T})$ is a cell. Putting
these together, one obtains

\begin{theorem}[Hazel\cite{ha}, Leibon\cite{Leibon}] \label{Hazel} There is a real analytic
cell decomposition
$$T_{hp}(S, V)= \cup_{[\mathcal{T}]}D_c(\mathcal{T})$$
invariant under the action of the mapping class group where the
union is over all isotopy classes $[\mathcal{T}]$ of
triangulations of $(S,V)$.
\end{theorem}



\section{Diffeomorphism between two Teichm\"uller spaces}

One of the main tools used in our proof is the decorated
Teichm\"uller space theory developed by R. Penner \cite{penner}.
See also \cite{be}, \cite{guoluo} and \cite{glsw} for a discussion
of Delaunay triangulations of decorated metrics.

Recall that $S$ is a closed connected surface and $V=\{v_1, ...,
v_n\} \subset S$ and let $\Sigma=S-V$. We assume $n \geq 1$ and
the Euler characteristic $\chi(\Sigma) <0$.  A \it decorated
hyperbolic metric \rm is a complete hyperbolic metric $d$ of
finite area on $\Sigma$ together with a horoball $H_i$ at the i-th
cusp for each $v_i$. The decorated metric will be written as a
pair $(d, w)$ where $w=(w_1, ..., w_n) \in \R^n_{>0}$ so that
$w_i$ is the length of the horocycle $\partial H_i$. The decorated
Teichm\"uller space, denoted by $T_D(\Sigma)$, is the space of all
decorated metrics on $\Sigma$ modulo isometries homotopic to the
identity and preserving decorations.  For a given triangulation
$\T$ of $(S,V)$, let $\Psi_{\T}: \R_{>0}^E \to T_D(\Sigma)$ be the
$\lambda$-length coordinate (see \cite{penner}) and let $D(\T)$ be
the set of all decorated hyperbolic metrics $(d,w)$ in
$T_D(\Sigma)$ so that $\T$ is isotopic to a Delaunay triangulation
of $(d,w)$. See \cite{penner} or \cite{glsw} for details.

Fix a triangulation $\T$ of $(S,V)$, we have two coordinate maps
$\Phi^{-1}_{\T}: P(\T) \to \R^{E(\T)}$ and  $\Psi_{\T}: \R^{E(\T)}
\to T_D(S,V)$. Consider the smooth embedding $A_{\T}: P(\T) \to
T_D(\Sigma)$ defined by $\Psi_{\T} \circ \Theta \circ
\Phi_{\T}^{-1}$, where $\Theta: \R^{E(\T)} \to \R^{E(\T)}$ sends
$(x_0,x_1,x_2,...)$ to
$(\sinh(x_0/2),\sinh(x_1/2),\sinh(x_2/2),...)$, i.e.,
$\Theta(x)(e) =\sinh(x(e)/2)$.

\begin{theorem} For each triangulation $\T$ of $(S,V)$,
$A_{\T}|_{D_c(\T)}$ is a real analytic diffeomorphism  from $D_c(\T)$ onto $D(\T)$.
\end{theorem}

\begin{proof}
To see that $A_{\T}$ maps $D_c(\T)$ bijectively onto $D(\T)$, it
suffices to show that $\Theta \circ \Phi_{\T}^{-1}(D_c(\T)) =
\Psi_{\T}^{-1}(D(\T))$.

The space $ \Psi_{\T}^{-1}(D(\T))$ can be characterized as
follows. For each edge $e$ in $(S, \T)$ with a decorated
hyperbolic metric $(d,w)$, let $a,a'$ be the two angles facing $e$
and $b,b', c, c'$ be the angles adjacent to the edge $e$.  Then
$\T$ is Delaunay in the metric $(d,w)$ if and only if for each
edge $e \in E(\T)$ (see \cite{penner}, or \cite{guoluo}),
\begin{equation} \label{fml:penner delau}
 a+a' \leq b+b'+c+c'.
\end{equation}

Let $t$ and $t'$ be the triangle adjacent to $e$ and $e, e_1, e_2$
be edges of $t$ and $e, e_3, e_4$ be the edges of $t'$. Let the
$\lambda$-length of $e$ be $\lambda_0$ and the $\lambda$-length of
$e_i$ be $\lambda_i$. Recall the cosine law for decorated ideal
triangles \cite{penner} states that $\alpha =\frac{x}{yz}$ where
$\alpha$ is the angle (i.e., the length of the horocyclic arc) and
$x,y,z$ are the $\lambda$-lengths so that $x$ faces $\alpha$.
Using it, one sees that (\ref{fml:penner delau}) is equivalent to

\begin{equation} \label{fml:penner delau 2}
\frac{\lambda_0}{\lambda_1 \lambda_2}+\frac{\lambda_0}{\lambda_3 \lambda_4}
 \leq
\frac{\lambda_1}{\lambda_0 \lambda_2}+\frac{\lambda_2}{\lambda_0 \lambda_1}+\frac{\lambda_3}{\lambda_0 \lambda_4}+\frac{\lambda_4}{\lambda_0 \lambda_3},
\end{equation}
for each $e \in E(\T)$.

Rearranging terms, we see (\ref{fml:penner delau 2}) is equivalent
to
\begin{equation} \label{fml:penner delau 3}
0 \leq  \frac{\lambda_1 ^2+ \lambda_2^2-\lambda_0^2}{\lambda_1 \lambda_2} +
\frac{\lambda_3^2 + \lambda_4^2 - \lambda_0^2}{\lambda_3 \lambda_4},
\end{equation}
for each $e \in E(\T)$.

Therefore, $$ \Psi_{\T}^{-1}(D(\T))=\{ (\lambda_0,\lambda_1,...,
\lambda_{|E|}) \in \R_{>0}^E | \text{ (\ref{fml:penner delau 3})
holds at each edge $e\in E(\T)$}\}.$$

By Theorem \ref{9} and proposition \ref{thm:delaunay condition},
the characterization of a hyperbolic polyhedral metric $d$ which
is Delaunay in $\T$ in terms of the length coordinate $x
=\Phi_{\T}^{-1}(d)$ is as follows. Take an edge $e \in E(\T)$ and
let $t$ and $t'$ be the triangles adjacent to $e$ so that $e, e_1,
e_2$ are edges of $t$ and $e, e_3, e_4$ are the edge of $t'$.
Suppose the length of $e$ (in $d$) is $x_0$ and the length of
$e_i$ is $x_i$, $i=1,...,4$. Then, by Proposition
\ref{thm:delaunay condition},
\begin{equation}\label{fml:empty circle}
0 \leq
\frac{\sinh^2(x_1/2)+\sinh^2(x_2/2)-\sinh^2(x_0/2)}{\sinh(x_1/2)\sinh(x_2/2)}+\frac{\sinh^2(x_3/2)+
\sinh^2(x_4/2)-\sinh^2(x_0/2)}{\sinh(x_3/2)\sinh(x_4/2)}
\end{equation}
holds for each edge $e \in E(\T)$.

This shows that
\begin{align*}
&\Phi_{\T}^{-1}(D_c(\T))=\{ x \in \R_{>0}^E | \text{
(\ref{fml:empty circle}) holds for  $e \in E$, and
(\ref{fml:triangle inequality}) holds for each triangle}\}
\end{align*}
where
\begin{equation}\label{fml:triangle inequality}
x(e_i)+x(e_j)> x(e_k),  \quad \text{ $e_i, e_j, e_k$ form edges of
a triangle in $\T$}. \end{equation}

Now inequality (\ref{fml:penner delau 3}) is the same as
(\ref{fml:empty circle}) by taking $\lambda_i$ to be
$\sinh(x_i/2)$ for each $i$. This shows $\Theta \circ
\Phi_{\T}^{-1}(D_c(\T)) \subset \Psi_{\T}^{-1}(D(\T))$. On the
other hand, corollary  \ref{12} implies that for each $\lambda \in
\Psi_{\T}^{-1}(D(\T))$ and a triangle of edges $e_i, e_j, e_k$, we
have $x(e_i)+x(e_j)> x(e_k)$ where $x(e)=2\sinh^{-1}(\lambda(e))$,
i.e., condition (\ref{fml:triangle inequality}) is a consequence
of (\ref{fml:empty circle}).
Therefore $\Theta \circ \Phi_{\T}^{-1}(D_c(\T)) =
\Psi_{\T}^{-1}(D(\T))$.

Finally, since  $\Phi_{\T}$, $\Psi_{\T}$ and $\Theta$ are real
analytic diffeomorphisms and $A_{\T}=\Psi_{\T}\circ \Theta  \circ
\Phi_{\T}^{-1}$ and $A^{-1}_{\T} =\Phi_{\T} \circ \Theta^{-1}
\circ \Psi^{-1}_{\T}$, we see that $A_{\T}$ is a real analytic
diffeomorphism.
\end{proof}

\subsection{The Ptolemy identity and diagonal switch}

Let $Q$ be a convex quadrilateral $Q$ in the Euclidean plane
$\mathbf E^2$, or the hyperbolic plane $\mathbf H^2$ or the
2-sphere $\mathbf S^2$ so that its edges are $a,b,a',b'$ counted
cyclically and its diagonals are $c,c'$. We say $Q$ is \it cyclic
\rm if it is circumscribed to a circle in $\mathbf E^2$, or
$\mathbf S^2$, or a curve of constant geodesic curvature in
$\mathbf H^2$. Let $l(e)$ to be the length of an edge $e$.

The classical Ptolemy theorem states that a Euclidean
quadrilateral $Q$ is cyclic if and only if the following holds
$$ l(a)l(a')+l(b)l(b')=l(c)l(c').$$

In the 19-th century, Jean Darboux and Ferdinand Frobenius proved
that a spherical quadrilateral $Q$ is cyclic if and only if
$$ \sin(\frac{l(a)}{2})\sin(\frac{l(a')}{2})+
\sin(\frac{l(b)}{2})\sin(\frac{l(b')}{2})=\sin(\frac{l(c)}{2})\sin(\frac{l(c')}{2}).$$

The hyperbolic case was established by T. Kubota in 1912 \cite{k}.
He proved,

\begin{prop}[Kubota]\label{kubota} A hyperbolic quadrilateral $Q$ is
inscribed to a curve of constant geodesic curvature in $\mathbf
H^2$ if and only if
\begin{equation}\label{h-ptolemy}
\sinh(\frac{l(a)}{2})\sinh(\frac{l(a')}{2})+
\sinh(\frac{l(b)}{2})\sinh(\frac{l(b')}{2})=\sinh(\frac{l(c)}{2})\sinh(\frac{l(c')}{2}).
\end{equation}

\end{prop}

Penner's Ptolemy identity \cite{penner} also takes the same form.
Namely, if $Q$ is a decorated ideal quadrilateral in $\mathbf H^2$
so that the $\lambda$-lengths of the its edges are $A,B,A', B'$
counted cyclically and its diagonal are $C,C'$, then
\begin{equation}\label{pid} AA'+BB'=CC'.
\end{equation}

The most remarkable feature of these theorems is that all
equations take the same form as $xx'+yy'=zz'$ which we will call
the Ptolemy identity. The Ptolemy identity also plays the key role
for cluster algebras associated to surfaces \cite{fst}.

The relationship between the Ptolemy identity and the diagonal
switch operation on Delaunay triangulations is the following. If
$\T$ and $\T'$ are two Delaunay triangulations of a Euclidean (or
hyperbolic or spherical) polyhedral surface $(S,V,d)$ so that they
are related by a diagonal switch from edge $e$ to edge $e'$, then
the change of the lengths from $l(e)$ and $l(e')$ is governed by
one of the Ptolemy identities listed above.

Casey's generalization of Ptolemy's theorem is another direction
where Ptolemy identity plays a key role. Furthermore, Casey's
theorem is known to be true for Euclidean, hyperbolic, spherical
and even Minkowski planes. In \cite{gl2}, we will exam the related
discrete conformality in the new setting.

\subsection{A globally defined diffeomorphism}

\begin{theorem}\label{17} Suppose $\T$ and $\T'$ are two triangulations of
$(S,V)$ so that $D_c(\T) \cap D_c(\T') \neq \emptyset$. Then

\begin{equation}\label{fml:gluing}
A_{\T}|_{D_c(\T) \cap
D_c(\T')} = A_{\T'}|_{D_c(\T) \cap D_c(\T')}.
\end{equation}

 In particular, the gluing of these $A_{\T}|_{D_c(\T)}$
mappings produces a homeomorphism $A=\cup_{\T} A_{\T}|_{D_c(\T)}:
T_{hp}(S,V) \to T_D(\Sigma)$ such that $A(d)$ and $A(d')$ have the
same underlying hyperbolic structure if and only if $d$ and $d'$
are discrete conformal.
\end{theorem}

\begin{proof} Suppose $d \in D_c(\T) \cap D_c
(\T')$, i.e., $\T$ and $\T'$ are both Delaunay in the hyperbolic
polyhedral metric $d$. Then by proposition \ref{134} there exists
a sequence of triangulations $\T_1 =\T, \T_2, ..., T_k =\T'$ on
$(S,V)$ so that each $\T_i$ is Delaunay in $d$ and $\T_{i+1}$ is
obtained from $\T_i$ by a diagonal switch. In particular,
$A_{\T}(d) =A_{\T'}(d)$ follows from $A_{\T_i}(d)=A_{\T_{i+1}}(d)$
for $i=1,2,..., k-1$.
 Thus, it suffices to show $A_{\T}(d) =A_{\T'}(d)$ when $\T'$ is obtained from
$\T$ by a diagonal switch along an edge $e$. This is the same as
showing $\Psi_{\T}^{-1} \Psi_{\T'} =\Theta
\Phi_{\T}^{-1}\Phi_{\T'} \Theta^{-1}$ at the point $x
=\Psi_{\T'}^{-1}(d)$. On the other hand, $\Psi_{\T}^{-1}
\Psi_{\T'}(x)$ and $\Theta \Phi_{\T}^{-1}\Phi_{\T'}
\Theta^{-1}(x)$ have the same coordinate except at the $e$ edge of
diagonal switch. For the edge $e$, the two coordinates  are the
same due to the Penner's Ptolemy identity (\ref{pid}) (for
$\Psi_{\T}^{-1} \Psi_{\T'}$) and Kubota's Ptolemy identity
(\ref{h-ptolemy}) (for $\Phi_{\T}^{-1}\Phi_{\T'}$). These two
identities differ by a change of variable $t \to
\sinh(\frac{t}{2})$ which corresponds to $\Theta$. Therefore,
$A_{\T}(d) =A_{\T'}(d)$.

  Taking the inverse, we obtain

\begin{equation}\label{fml:gluing 2}
A_{\T}^{-1}|_{D(\T) \cap D(\T')} =
A_{\T'}^{-1}|_{D(\T) \cap D(\T')}.
\end{equation}

\begin{lemma}
\begin{itemize}
\item[(a)] $D_c(\T) \cap
D_c(\T') \neq \emptyset$ if and only if $D(\T) \cap D(\T') \neq
\emptyset$.
\item[(b)] The gluing map $A =\cup_{\T} A_{\T}|_{D_c(\T)}: T_c
\to T_D$ is a homeomorphism invariant under the action of the
mapping class group.
\end{itemize}
\end{lemma}

\begin{proof} By (\ref{fml:gluing}) and (\ref{fml:gluing 2}), the maps  $A =\cup_{\T} A_{\T}|_{D_c(\T)}: T_c \to T_D$
and $B =\cup_{\T} A^{-1}_{\T}|_{D(\T)}: T_D \to T_c$ are well
defined and  continuous. Since $A(D_c(\T) \cap D_c(\T') ) \subset
D(\T) \cap D(\T')$ and $B(D(\T) \cap D(\T')) \subset D_c(\T) \cap
D_c(\T')$, part (a) follows. To see part (b), by Penner's result
\cite{penner} that $T_D =\cup_{\T} D(\T)$, the map $A$ is onto. To
see $A$ is injective, suppose $x_1 \in D_c(\T_1), x_2 \in
D_c(\T_2)$ so that $A(x_1) =A(x_2) \in D(\T_1) \cap D(\T_2)$.
Apply (\ref{fml:gluing 2}) to $A^{-1}_{\T_1}|, A^{-1}_{\T_2}|$ on
the set $D(\T_1) \cap D(\T_2)$ at the point $A(x_1)$, we conclude
that $x_1=x_2$. This shows that $A$ is a bijection with inverse
$B$. Since both $A$ and $B$ are continuous, $A$ is a
homeomorphism. $\square$
\end{proof}

Now if $d$ and $d'$ are two discrete conformally equivalent
hyperbolic polyhedral metrics, then  $A(d)$ and $A(d')$ are of the
form $(p, w)$ and $(p, w')$ due to the definitions.  Indeed, if
$d$ and $d'$ are related by condition (b) in definition
\ref{12345}, then the discrete conformality translates to the
change of decoration without changing the hyperbolic metric. (This
is the same proof as in \cite{glsw}, lemma 3.1).  If $d$ and $d'$
are related by condition (c) in definition \ref{12345}, then the
two triangulations $\T_i$ and $\T_{i+1}$ are both Delaunay in
$[d]$. Therefore, in this case, $A(d)=A(d')$.

On the other hand, if two hyperbolic cone metrics $d, d'$ satisfy
that $A(d)$ and $A(d')$ are of the form $(p, w)$ and $(p, w')$,
consider a generic smooth path $\gamma(t)=(p, w(t)), t \in [0,1]$,
in $T_D(\Sigma)$ from $(p,w)$ to $(p, w')$ so that $\gamma(t)$
intersects the cells $D(\T)$'s transversely. This implies that
$\gamma$ passes through a finite set of cells $D(\T_i)$ and $\T_j$
and $\T_{j+1}$ are related by a diagonal switch. Let
 $t_0=0<... <t_m=1$ be a partition of $[0,1]$ so that $\gamma([t_i,
t_{i+1}]) \subset D(\T_i)$. Say $d_i$ is the hyperbolic polyhedral
metric so that $A(d_i) =\gamma(t_i) \in D(\T_i) \cap D(\T_{i+1})$,
$d_1=d$ and $d_m=d'$. Then by definition, the sequences $\{d_1,
..., d_m\}$ and the associated Delaunay triangulations $\{\T_1,
..., T_m\}$ satisfy the definition of discrete conformality for
$d,d'$.

\end{proof}

\begin{theorem} The homeomorphism $A: T_{hp}(S,V) \to T_D(\Sigma)$
is a $C^1$ diffeomorphism.
\end{theorem}

\begin{proof} It suffices to show that for a point $d \in D_c(\T)
\cap D_c(\T')$, the derivatives $DA_{\T}(d)$ and $DA_{\T'}(d)$
are the same. Since  both $\T$ and $\T'$ are Delaunay in $d$ and
are related by a sequence of Delaunay triangulations (in $d$)
$\T_1=\T, \T_2, ..., \T_k =\T'$, $DA_{\T}(d)=DA_{\T'}(d)$ follows
from $DA_{\T_i}(d) =DA_{\T_{i+1}}(d)$ for $i=1,2,..., k-1$.
Therefore, it suffices to show $DA_{\T}(d) = DA_{\T'}(d)$ when
$\T$ and $\T'$ are related by a diagonal switch at an edge $e$. In
the coordinates $\Phi_{\T}$ and $\Psi_{\T}$, the
 fact that $DA_{\T}(d) = DA_{\T'}(d)$  is equivalent to  the following smoothness
question on the diagonal lengths.

\begin{lemma}\label{thm:derivative} Suppose $Q$ is a convex hyperbolic quadrilateral
whose four edges are of lengths $x,y,z,w$ (counted cyclically) and
the length of a diagonal is $a$. Suppose $A(x,y,z,w,a)$ is the
length of the other diagonal and $B(x,y,z,w,a)=2
\sinh^{-1}(\frac{s(x)s(z)+s(y)s(w)}{s(a)})$ where $s(t)=
\sinh(\frac{t}{2})$. If a point $(x,y,z,w,a)$ satisfies
$A(x,y,z,w,a)=B(x,y,z,w,a)$, i.e., $Q$ is inscribed in a curve of
constant geodesic curvature, then $DA(x,y,z,w,a)=DB(x,y,z,w,a)$
where $DA$ is the derivative of $A$.
\end{lemma}

Due to the lengthy proof of this lemma, we defer it to the
appendix.

\end{proof}

\begin{corollary}\label{thm:conformal class} For a given hyperbolic
 polyhedral metric $d$ on $(S,V)$, the set of all Teichm\"uller equivalence classes of hyperbolic metrics
  on $(S, V)$ which are discrete conformal to $d$ is $C^1$-diffeomorphic
  to $ \mathbb{R}^{|V|}.$
\end{corollary}

\section{Discrete Uniformization for Hyperbolic Polyhedral Metrics}
This section proves theorem \ref{thm:hyperbolic} which is the main
result of this paper.

By Corollary \ref{thm:conformal class}, Theorem
\ref{thm:hyperbolic} is equivalent to a statement about the
composition map of the discrete curvature map $K$ and $(A|)^{-1}$
defined on $\{p\} \times \R_{>0}^n \subset T_D(\Sigma)$ for any $p
\in T(\Sigma)$. Here $K: T_{hp}(S,V) \to (-\infty, 2\pi)^n$ is the
map sending a metric $d$ to its discrete curvature $K_d$.  Let us
make a change of variables from $w=(w_1, ..., w_n) \in \R_{>0}^n$
to $u=(u_1, ..., u_n) \in \R^n$ where $u_i =\ln(w_i)$. We write
$w=w(u)$. For a given $p \in T(\Sigma)$, define  $F$ to be the
composition of $K$ and   $(A|)^{-1}$ from $ \R^n$ to $(-\infty,
2\pi)^n$ by
\begin{equation}\label{fml:curv}
F(u) =K_{A^{-1}(p, w(u))}.
\end{equation}

 By the Gauss-Bonnet theorem, the
image $F(u)$ lies in the open subset $\mathbf P=\{ x \in (-\infty,
2\pi)^n | \sum_{i=1}^n x_i
>2\pi \chi(S)\}$ of $\R^n$. Theorem
\ref{thm:hyperbolic} is equivalent to that $F: \R^n \to \mathbf P$
is a bijection. We will show a stronger statement that $F$ is a
homeomorphism.

For simplicity, we use $s(t)$ to denote the function
$\sinh(\frac{t}2)$.

\subsection{Injectivity of $F$ }

Since $A$ is a $C^1$ diffeomorphism and the discrete curvature $K:
\T_{hp}(S,V) \to \R^V$ is real analytic, hence the map $F$ is
$C^1$ smooth.

On the other hand, we have,

\begin{theorem}[Akiyoshi \cite{ak}] 
For any finite area complete hyperbolic metric $p$ on $\Sigma$,
there are only finitely many isotopy classes of triangulations
$\T$ so that $([p] \times \R_{>0}^n) \cap D(\T) \neq \emptyset$.
\end{theorem}

Let $\T_i$, $i=1,..., k$, be the set of all triangulations so that
$(\{p\} \times \R^n) \cap D(\T_i) \neq \emptyset$ and $\{p\}
\times \R^n \subset \cup_{i=1}^k D(\T_i)$.

\begin{lemma}  Let $\phi: \R^n \to \{p\}\times \R^n$ be $\phi(x)=(p, x)$ and
 $U_i =\phi^{-1}((\{p\} \times \R^n) \cap D(\T_i)) \subset \R^n$ and $J =\{ i | $
\text{$int(U_i) \neq \emptyset$}\}. Then $ \R^n =\cup_{i \in J}
U_i$ and $U_i$ is real analytic diffeomorphic to a convex polytope
in $ \R^n$.
\end{lemma}

\begin{proof} By definition, both $\{p\} \times \R^n$ and $D(\T_i)$
are closed and semi algebraic in $T_D(\Sigma)$. Therefore $U_i$ is
closed in $\R^n$ and is diffeomorphic under $w=w(u)$ to a
semi-algebraic set. Now by definition, $Y: =\cup_{i \in J} U_i$ is
a closed subset of $ \R^n$ since $U_i$ is closed. If $Y \neq
\R^n$, then the complement $ \R^n -Y$ is a non-empty open set
which is diffeomorphic under $w=w(u)$ to a finite union of real
algebraic sets of dimension less than $n$. This is impossible.

Finally, we will show that for any triangulation $\T$ of $(S,V)$
and $p \in T(\Sigma)$, the intersection $U =\phi^{-1}( (\{p\} \times R^n)
\cap D(\T))$ is real analytically diffeomorphic to a convex
polytope in a Euclidean space. In fact $\Psi_{\T}^{-1}(U) \subset
\R^{E(\T)}$ is  real analytically diffeomorphic to a convex
polytope. To this end, let $b =\Psi_{\T}(p, (1,1,....,1))$. By
definition, $\Psi_{\T}^{-1}(U)$ is give by
\begin{align*}
&\{ x \in \R^{E(\T)}_{>0}| \exists \lambda \in \R_{>0}^V,
\sinh(x(e)/2) =b(e) \lambda(v_1) \lambda(v_2),
\partial e=\{v_1, v_2\}, \\
&\hspace{120pt} \text{Delaunay condition (\ref{eq:algdel}) holds
for $x$}\}.
\end{align*}

We claim that the Delaunay condition (\ref{eq:algdel}) consists of
linear inequalities in the variable $\delta: V\to \R_{>0}$ where
$\delta(v) = \lambda(v)^{-2}$. Indeed, suppose the two triangles
adjacent to the edge $e=(v_1, v_2)$ have vertices $v_1, v_2, v_3$
and $v_1, v_2, v_4$. Let $x_{ij}$ (respectively $b_{ij}$) be the
value of $x$ (respectively $b$) at the edge joining $v_i, v_j$,
and $\lambda_i=\lambda(v_i)$ and let $s(t)$ be the function
$\sinh(\frac{t}{2})$. By definition, $s(x_{ij})=b_{ij} \lambda_i
\lambda_j$. The Delaunay condition (\ref{eq:algdel}) at the edge
$e=(v_1v_2)$ says that

\begin{equation*} \frac{s(x_{12})^2}{s(x_{31})s(x_{32})} +
\frac{s(x_{12})^2}{s(x_{41})s(x_{42})} \leq
\frac{s(x_{31})}{s(x_{32})} + \frac{s(x_{32})}{s(x_{31})}
+\frac{s(x_{41})}{s(x_{42})} +\frac{s(x_{42})}{s(x_{41})}
\end{equation*}
It is the same as, using
$s(x_{ij})=b_{ij}\lambda_i \lambda_j$,

$$ c_3 \frac{\lambda_1 \lambda_2}{\lambda_3^2} + c_4 \frac{\lambda_1\lambda_2}{\lambda_4^2} \leq
c_1\frac{\lambda_2}{\lambda_1} + c_2\frac{\lambda_1}{\lambda_2},$$
where $c_i$ is some constant depending only on $b_{jk}$'s.
Dividing above inequality by $\lambda_1 \lambda_2$ and using
$\delta_i=\lambda_i^{-2}$, we obtain

\begin{equation}\label{fml:delta}
c_3 \delta_3 + c_4 \delta_4 \leq c_1 \delta_1 + c_2 \delta_2
\end{equation}
at each edge $e\in E(\T)$. This shows for $b$ fixed, the set of
all possible values of $\delta$ form a convex polytope $\mathbf Q$
defined by (\ref{fml:delta}) at all edges and $\delta(v)>0$ at all
$v \in V$. On the other hand, by definition, the map from $\mathbf
Q$ to $\Psi_{\T}^{-1}(U)$ sending $\delta$ to $x=x(\delta)$ given
by $x(vv') =2 \sinh^{-1}(\frac{b(vv')}{\sqrt{ \delta(v)
\delta(v')}})$ is a real analytic diffeomorphism. Thus the result
follows.
\end{proof}

Write $F=(F_1, ..., F_n)$ which is $C^1$ smooth. The work of
Bobenko-Pinkall-Springborn (\cite{bps}, proposition 5.1.5) shows
that

\begin{itemize}
\item[(a)] $F_j|_{U_h}$ is real analytic so that
$\frac{\partial F_i}{\partial u_j} =\frac{\partial F_j}{\partial
u_i}$ in $U_h$ for all $h \in J$,
\item[(b)] the Hessian matrix
$[\frac{\partial F_i}{\partial u_j}]$ is positive definite
on each $U_h$.
\end{itemize}

Therefore, the 1-form $\eta =\sum_{i} F_i(u) du_i$ is a $C^1$
smooth 1-form on $\R^n$ so that $d\eta =0$ on each $U_h, h\in J$.
This implies that $d \eta =0$ in $\R^n$. Hence the integral
\begin{equation}\label{eq:w}W(u) =\int_{0}^u \eta \end{equation}
is a well defined $C^2$ smooth
function on $\R^n$ so that its Hessian matrix is positive
definite. Therefore, $W$ is convex in $\R^n$ so that its gradient
$\nabla W=F$. Now $F$ is injective due to the following well known
lemma,

\begin{lemma} If $W: \Omega \to \R$ is a  $C^1$-smooth strictly convex
function on an open convex set $\Omega \subset \R^m$, then its
gradient $\nabla W: \Omega \to \R^m$ is an embedding.
\end{lemma}

\subsection{The map $F$ is onto}

Since both $\R^n$ and $\mathbf P =\{ x \in (-\infty, 2\pi)^n |
\sum_{i=1}^n x_i
>2\pi \chi(S)\}$ are connected manifolds of
dimension $n$ and $F$ is injective and continuous, it follows that
$F(\R^n)$ is open in $\mathbf P$. To show that $F$ is onto, it
suffices to prove that $F(\R^n)$ is closed in $\mathbf P$.

To this end, take a sequence $\{u^{(m)}\}$ in $\R^n$ which leaves
every compact set in $\R^n$. We will show that $\{F(u^{(m)})\}$
leaves each compact set in $\mathbf P$. By taking subsequences, we
may assume that for each index $i=1,2,...,n$, the limit $\lim_{m}
u^{(m)}_i = t_i$ exists in $[-\infty, \infty]$. Furthermore, by
Akiyoshi's theorem that the space $p \times \R^n$ is in the union
of a finite number of Delaunay cells $D(\T)$, we may  assume,
after taking another subsequence, that the corresponding
hyperbolic polyhedral metrics $d_m= A^{-1}(p, w(u^{(m)}))$ are in
$D(\T)$ for one triangulation $\T$. We will calculate in the
length coordinate $\Phi_{\T}$ below.

Since $u^{(m)}$ does not converge to any vector in $\R^n$, there exists $t_i=\infty$ or $-\infty$.
Let us label vertices $v \in V$ by \it black
\rm and \it white \rm as follows. The vertex $v_i$ is black if and
only if $t_i=-\infty$ and all other vertices are white.

\begin{lemma}\label{thm:limit}
(a) There does not exist a triangle $\tau \in \T$ with
exactly two white vertices.

(b) If $\Delta v_1v_2v_3$ is a triangle with exactly one white
vertex at $v_1$, then the inner angle of the triangle at $v_1$
converges to $0$ as $m \to \infty$ in the metrics $d_m$.
\end{lemma}

\begin{proof}
To see (a), suppose otherwise, using the $\Phi_{\T}$ length
coordinate, we see the given assumption is equivalent to
following. There exists a hyperbolic triangle of lengths
$l_1^{(m)},l_2^{(m)},l_3^{(m)}$ such that
$s(l_i^{(m)})=s(a_i)e^{u^{(m)}_j +u^{(m)}_k}$,
$\{i,j,k\}=\{1,2,3\}$, where $\lim_{m} u^{(m)}_i >- \infty$ for
$i=2,3$ and $\lim_{m} u^{(m)}_1 =-\infty$. By applying
$\sinh(t/2)$ to the triangle inequality
$l_2^{(m)}+l_3^{(m)}>l_1^{(m)}$  and using angle sum formula for
$\sinh$, we obtain
$$s(l_2^{(m)})\sqrt{1+s(l_3^{(m)})^2}+s(l_3^{(m)})\sqrt{1+s(l_2^{(m)})^2}>s(l_1^{(m)}).$$
Thus
\begin{align*}
&s(a_2) e^{u^{(m)}_1+u^{(m)}_3}\sqrt{1+s(a_3)^2 e^{2u^{(m)}_1+2u^{(m)}_2}}+s(a_3) e^{u^{(m)}_1+u^{(m)}_2}\sqrt{1+s(a_2)^2 e^{2u^{(m)}_1+2u^{(m)}_3}}\\
&>s(a_1) e^{u^{(m)}_2+u^{(m)}_3}.
\end{align*}

This is the same as
$$s(a_2)\sqrt{e^{-2u^{(m)}_2}+s(a_3)^2 e^{2u^{(m)}_1}}+ s(a_3)\sqrt{e^{-2u^{(m)}_3}+s(a_3)^2 e^{2u^{(m)}_1}}  > s(a_1) e^{-u^{(m)}_1}.$$
However, by the assumption, the right-hand-side tends to $\infty$
and the left-hand-side is bounded. The contradiction shows that
(a) holds.

To see (b), we use the same notation as in the proof of (a). Let
$\alpha_1^{(m)}$ be the inner angle at $v_1$ of the triangle
$\Delta v_1v_2v_3$ in $d_m$ metric. Our goal is to show $\lim_m
\alpha_1^{(m)}=0$.



Since the sequence of hyperbolic polyhedral metrics $\{d_m\}$ are
Delaunay in the same triangulation $\mathcal{T}$, by proposition
\ref{thm:cycle}, the three numbers $s(l_1^{(m)}),
s(l_2^{(m)}),s(l_3^{(m)})$ satisfy the triangle inequality.
Therefore, for each $m$, there is a Euclidean triangle whose sides
have lengths $s(l_1^{(m)}), s(l_2^{(m)}), s(l_3^{(m)})$. Since
$s(l_i^{(m)})=s(a_i)e^{u_j^{(m)}+u_k^{(m)}}$, this triangle is
similar to the Euclidean triangle $\Delta$ whose sides have
lengths $s(a_1)e^{-u_1^{(m)}}$, $s(a_1)e^{-u_2^{(m)}}$ and
$s(a_1)e^{-u_3^{(m)}}.$ By the assumption that $\lim_m u_1^{(m)}
> -\infty$ and $\lim_m u_2^{(m)}=-\infty$ and
$\lim_m u_3^{(m)}=-\infty$, the three edge lengths
$s(a_1)e^{-u_1^{(m)}}, s(a_1)e^{-u_2^{(m)}}, s(a_1)e^{-u_3^{(m)}}$
tend to $t \in \R$, $\infty$ and $\infty$ respectively. Therefore
the angle in the Euclidean triangle $\Delta$ opposite to the edge
of length $s(a_1)e^{-u_1^{(m)}}$ approaches $0$. By the cosine law
for Euclidean triangle, we obtain
$$\lim_m \frac{s(l_2^{(m)})^2+s(l_3^{(m)})^2-
s(l_1^{(m)})^2}{2s(l_2^{(m)})s(l_3^{(m)})} =1. $$

On the other hand, from Lemma \ref{thm:relation}, we have

$$\sin\frac{\alpha_2^{(m)}+\alpha_3^{(m)}-\alpha_1^{(m)}}2 \cdot \cosh\frac{l_1^{(m)}}2= \frac{s(l_2^{(m)})^2+s(l_3^{(m)})^2-s(l_1^{(m)})^2}{2s(l_2^{(m)})s(l_3^{(m)})}.$$

Also we have $\lim_m l_1^{(m)}=0$ due to  $\lim_m
u_2^{(m)}=-\infty$ and $\lim_m u_3^{(m)}=-\infty$. Hence
$$\lim_m \sin\frac{\alpha_2^{(m)}+\alpha_3^{(m)}-\alpha_1^{(m)}}2=1.$$
It is equivalent to $$\lim_m
(\alpha_2^{(m)}+\alpha_3^{(m)}-\alpha_1^{(m)}) =\pi \geq \lim_m
(\alpha_2^{(m)}+\alpha_3^{(m)}+\alpha_1^{(m)}) . $$

Thus $$\lim_m \alpha_1^{(m)} \leq 0.$$

Hence $$\lim_m \alpha_1^{(m)} = 0.$$

\end{proof}

We now finish the proof of $F(\R^n)=\mathbf P$ as follows.

Case 1. All vertices are white. There exists $t_i=\infty.$ Let $\triangle v_iv_jv_k$ be a triangle at vertex $v_i$. There exists a hyperbolic triangle of lengths $l_i^{(m)},l_j^{(m)},l_k^{(m)}$ such that $s(l_i^{(m)})=s(a_i)e^{u^{(m)}_j +u^{(m)}_k}$ (similar formulas hold for $l_j^{(m)}$ and $l_k^{(m)}$). Then $\lim_m l_j^{(m)}= \lim_m l_k^{(m)}=\infty.$
Let $\alpha_i^{(m)}$ be the inner angle at $v_i$.
By the cosine rule,
                   \begin{align*}
\lim_m \cos \alpha_i^{(m)} &=\lim_m \frac{-\cosh l_i^{(m)}+\cosh l_j^{(m)} \cosh l_k^{(m)}}{\sinh l_j^{(m)} \sinh l_k^{(m)}} \\
                           &=\lim_m \frac{-\cosh l_i^{(m)}+\cosh l_j^{(m)} \cosh l_k^{(m)}}{\cosh l_j^{(m)} \cosh l_k^{(m)}}
                           \cdot \lim_m \frac{\cosh l_j^{(m)} \cosh l_k^{(m)}}{\sinh l_j^{(m)} \sinh l_k^{(m)}} \\
                           &=\lim_m \frac{-\cosh l_i^{(m)}+\cosh l_j^{(m)} \cosh l_k^{(m)}}{\cosh l_j^{(m)} \cosh l_k^{(m)}} \\
                           &=-\lim_m \frac{\cosh l_i^{(m)}}{\cosh l_j^{(m)} \cosh l_k^{(m)}}+1 \\
                           &=-\lim_m \frac{2s(l_i^{(m)})^2+1}{(2s(l_j^{(m)})^2+1)(2s(l_k^{(m)})^2+1)}+1 \\
                           &=-\lim_m \frac{2s(l_i^{(m)})^2}{(2s(l_j^{(m)})^2+1)(2s(l_k^{(m)})^2+1)}+1 \\
                           &=-\lim_m \frac{2s(a_i)^2e^{2u_j^{(m)}+2u_k^{(m)}}}{(2s(a_j)^2e^{2u_i^{(m)}+2u_k^{(m)}}+1)(2s(a_k)^2e^{2u_i^{(m)}+2u_j^{(m)}}+1)}+1 \\
                           &=-\lim_m \frac{2s(a_i)^2}{(2s(a_j)^2e^{2u_i^{(m)}}+e^{-2u_k^{(m)}})(2s(a_k)^2e^{2u_i^{(m)}}+e^{-2u_j^{(m)}})}+1 \\
                           &=1.
\end{align*}

Therefore each inner angle at $v_i$ approaches $0$. The curvature
of $d_m$ at $v_i$ approaches $2\pi.$ This shows that $F(u^{(m)})$
tends to infinity of $\mathbf P$.

Case 2. All vertices are black. Then the length of each edge
approaches $0$. Each hyperbolic triangle approaches a Euclidean
triangle. The sum of the curvatures at all vertices approaches
$2\pi\chi(S)$. This shows that $F(u^{(m)})$ tends to infinity of
$\mathbf P$.

Case 3. There exist both white and black vertices. Since the
surface $S$ is connected, there exists an edge $e$ whose end
points $v, v_1$ have different colors. Assume $v$ is white and
$v_1$ is black. Let $v_1, ..., v_k$ be the set of all  vertices
adjacent to $v$ so that $v, v_i, v_{i+1}$ form vertices of a
triangle and let $v_{k+1}=v_1$. Now applying part (a) of Lemma
\ref{thm:limit} to triangle $\Delta vv_1v_2$ with $v$ white and
$v_1$ black, we conclude that $v_2$ must be black. Repeating this
to $\Delta vv_2v_3$ with $v$  white and $v_2$ black, we conclude
$v_3$ is black. Inductively, we conclude that all $v_i$'s, for
$i=1,2,..., k$, are black. By part (b) of Lemma \ref{thm:limit} ,
we conclude that the curvature of $d_m$ at $v$ tends to $2\pi$.
This shows that $F(u^{(m)})$ tends to infinity of $\mathbf P$.

Cases 1,2,3 show that $F(\R^n)$ is closed in $\mathbf P$.
Therefore $F(\R^n)=\mathbf P$.

\subsection{Discrete Yamabe flow}
Given $K^* \in (-\infty, 2\pi)^V$ so that $\sum_{v \in V}$$ K^*(v)
> 2\pi \chi(S),$ by the proof above, there exists $u^* \in \R^n$
so that $F(u^*)=K^*$. Furthermore, the function $F$ is the
gradient $\bigtriangledown W$ of a strictly convex function $W(u)$
defined on (\ref{eq:w}) on $\R^n$.

The discrete Yamabe flow with surgery is defined to be the
gradient flow of the strictly convex function $W^*(u)=W(u)
-\sum_{i=1}K^*_i u_i$.  This flow is a generalization of the
discrete Yamabe flow introduced in \cite{luo}. Since $F(u^*)=K^*$,
we see  $\bigtriangledown W^*(u^*) =0$, i.e., $W^*$ has a unique
minimal point $u^*$ in $\R^n$. It follows that the gradient flow
of $W^*$ converges to the minimal point $u^*$ as time approaches
infinity.

In the formal notation, the flow takes the form
$\frac{du_i(t)}{dt} = K_i-K^*_i$ and $u(0)=0$. The exponential
convergence of the flow can be established using exactly the same
method used for Theorem 1.4 of \cite{luo}.

\section{Algorithmic aspect of discrete conformality}
We will prove theorem 2 in this section.

Suppose $\alpha$ and $\alpha'$ are two hyperbolic (or Euclidean)
polyhedral metrics on $(S,V)$ given in terms of edge lengths in
two geodesic triangulations $\T$ and $\T'$, i.e., $l=
\Phi^{-1}_{\T}(\alpha)$ and $l' =\Phi^{-1}_{\T'}(\alpha')$ are two
vectors in $\R^{E(\T)}$ and $\R^{E(\T')}$. We will produce an
algorithm to decide if $d$ and $d'$ are discrete conformal using
the data $(\T, l)$ and $(\T', l')$.

There are two steps involved in the algorithm.

In the first step, using proposition \ref{134}(c), we may assume
that both $\T$ and $\T'$ are Delaunay in metrics $\alpha$ and
$\alpha'$ respectively. (The same also holds for Euclidean
polyhedral metrics. This is a well known fact from computational
geometry. See for instance \cite{bs}). Next, consider two
decorated hyperbolic metrics $(d,w)=A_{\T}(\alpha)$ and
$(d',w')=A_{\T'}(\alpha')$ with their respective Penner's
$\lambda$-coordinates $y=\Psi_{\T}^{-1}(d,w)$ and
$y'=\Psi_{\T'}^{-1}(d',w')$. By theorem \ref{17}, we see Theorem 2
follows from,

\begin{prop} Suppose  two decorated
hyperbolic metrics $(d,w)$ and $(d',w')$ in $T_{D}(\Sigma)$ are
given in terms of $\lambda$-lengths in two triangulations. There
exists an algorithm to decide if $d=d'$.
\end{prop}

\begin{proof} By the construction $y=\Psi_{\T}^{-1}(d,w)$ and $y'
=\Psi_{\T'}^{-1}(d',w')$ are the two
 $\lambda$-lengths.  Our goal is to use  $y$ and $y'$ to
decide if $d=d'$. There are two cases  according to $\T$ and $\T'$
are isotopic or not.

In the first case, $\T$ and $\T'$ are isotopic. Then it is known
by the work of Penner \cite{penner} that $d=d'$ if and only if the
associated Thurston's shear coordinates of $y$ and $y'$ are the
same. Here the shear coordinate $z$ of $y$ is defined to be
$z(e)=\frac{y(e_1) y(e_3)}{y(e_2)y(e_4)}$ with $e_1, e_2, e_3,
e_4$ being a (fixed) cyclically ordered edges of the quadrilateral
associated to $e$. Thus one can check algorithmically if $d=d'$
using $y$ and $y'$.

In the second case that  $\T$ and $\T'$ are not isotopic, we can
algorithmically produce $y''=\Psi_{\T}^{-1}(d',w')$ from $y'$ and
$\T'$. Indeed, a well known theorem of L. Mosher \cite{mosher}
says that there exists an algorithm to produce a finite set of
triangulations $\T_1 =\T', \T_2, ..., \T_k=\T$ so that $\T_{i+1}$
is obtained from $\T_i$ by a diagonal switch. Penner's Ptolemy
identity shows that one can compute algorithmically
$\Psi_{\T_{i+1}}^{-1}(d',w')$ from $\Psi_{\T_{i}}^{-1}(d', w')$.
Thus we can algorithmically compute the new $\lambda$-length
coordinate $y'' =\Psi_{\T}^{-1}(d',w')$ from
$y'=\Psi_{\T'}^{-1}(d',w')$. This reduces the problem to the first
case.
\end{proof}

\section{Appendix}

In the appendix we prove  Lemma \ref{thm:derivative}. 
Let $s(x)=\sinh\frac x2.$

\begin{lemma}[Fenchel \cite{Fenchel} page 118]\label{thm:radius}
Given a hyperbolic triangle with side lengths $a,b,c$, then
$$
\frac{(s(a)s(b)s(c))^2}{(s(a)+s(b)+s(c))(s(a)+s(b)-s(c))(s(b)+s(c)-s(a))(s(c)+s(b)-s(a))}
$$
equals
\begin{itemize}
\item $\frac14\sinh^2 r$ if the triangle has a compact
circumcircle of radius $r$, \item $\infty$ if the circumcircle is
a horocycle, \item $-\frac14 \cosh^2 D$ if the circumcircle is of
constant distance $D$ to a geodesic.
\end{itemize}
\end{lemma}

As a corollary we have,

\begin{lemma}\label{thm:sin rule}
Denote by $\alpha, \beta, \gamma$ the angles opposite to the sides
with lengths $a,b,c$. Then
$$\frac{\sinh a}{\sin \alpha}=2\zeta \cosh \frac a2 \cosh \frac b2 \cosh \frac c2,$$ where $\zeta$ equals
\begin{itemize}
\item $\tanh r$ if the triangle has a compact circumcircle of
radius $r$, \item $1$ if the circumcircle is a  horocycle, \item
$\coth D$ if the circumcircle  is of constant distance $D$ to a
geodesic.
\end{itemize}
\end{lemma}

\begin{proof} Assume that the triangle has a circumscribed circle of radius $r$. By using the cosine rule and Lemma \ref{thm:radius},
\begin{align*}
\sin \alpha&=(1-\cos^2 \alpha)^{\frac12}\\
           &=\frac{(-\cosh^2a-\cosh^2b-\cosh^2c+1+2\cosh a\cosh b\cosh c)^\frac12}{\sinh b \sinh c} \\
           &=\frac2{\sinh b \sinh c} \cdot \{4s(a)^2s(b)^2s(c)^2+\\
           & 2s(a)^2s(b)^2+2s(b)^2s(c)^2+2s(c)^2s(a)^2-s(a)^4-s(b)^4-s(c)^4\}^{\frac12}        \\
           &=\frac2{\sinh b \sinh c} \cdot \{4s(a)^2s(b)^2s(c)^2+\\
           & (s(a)+s(b)+s(c))(s(a)+s(b)-s(c))(s(b)+s(c)-s(a))(s(c)+s(b)-s(a))\}^{\frac12}   \\
           &=\frac2{\sinh b \sinh c} \cdot \{4s(a)^2s(b)^2s(c)^2+\frac{4s(a)^2s(b)^2s(c)^2}{\sinh^2r} \}^{\frac12} \\
           &=\frac4{\sinh b \sinh c} \cdot s(a)s(b)s(c)\frac{\cosh r}{\sinh r}.
\end{align*}

By taking limit with $r\to \infty,$ we can prove the lemma for the
case that the triangle has a horocyclic circumcircle.

Similar calculation can be used to prove the lemma for the case
that the triangle has a circumscribed equidistant curve.
\end{proof}



\begin{lemma}\label{thm:ptolemy} Let $a,b,c,d$ be the side lengths of a hyperbolic quadrilateral
 and $e,f$ the diagonal lengths so that $a,b,c,d$ are cyclically ordered edge lengths and edges of lengths
 $a,b,e$ form a triangle.
\begin{itemize}
\item[(i)] The vertices of this quadrilateral lie on a curve of
constant geodesic curvature. \item[(ii)] Ptolemy's formula holds:
$$s(e)s(f)=s(a)s(c)+s(b)s(d).$$ \item[(iii)]
\begin{equation}\label{fml:p1}
s(e)^2=(s(a)s(c)+s(b)s(d))\frac{s(a)s(d)+s(b)s(c)}{s(a)s(b)+s(c)s(d)},
\end{equation}
and
\begin{equation*}
s(f)^2=(s(a)s(c)+s(b)s(d))\frac{s(a)s(b)+s(c)s(d)}{s(a)s(d)+s(b)s(c)}.
\end{equation*}
\end{itemize}
\end{lemma}

\begin{proof}
\noindent (i)$\Longrightarrow$(ii). It was proved by T. Kubota
\cite{k}.

\noindent (ii)$\Longrightarrow$(i). It was proved by Joseph E.
Valentine \cite{v}, Theorem 3.4.

\noindent (iii)$\Longrightarrow$(ii). The product of the two
equations in (iii) produces the equation in (ii).

\noindent (i)$\Longrightarrow$(iii).

\noindent Case 1. When the vertices lie on a circle, it was proved
in \cite{GS11} (theorem 1, page 4).

\begin{figure}[ht!]
\labellist\small\hair 2pt \pinlabel $a$ at 210 423 \pinlabel $b$
at 313 482 \pinlabel $d$ at 309 363 \pinlabel $c$ at 403 434
\pinlabel $e$ at 258 411 \pinlabel $f$ at 366 423 \pinlabel
$\mathbb{R}$ at 108 273 \pinlabel $\mathbb{H}^2$ at 108 484

\endlabellist
\centering
\includegraphics[scale=0.4]{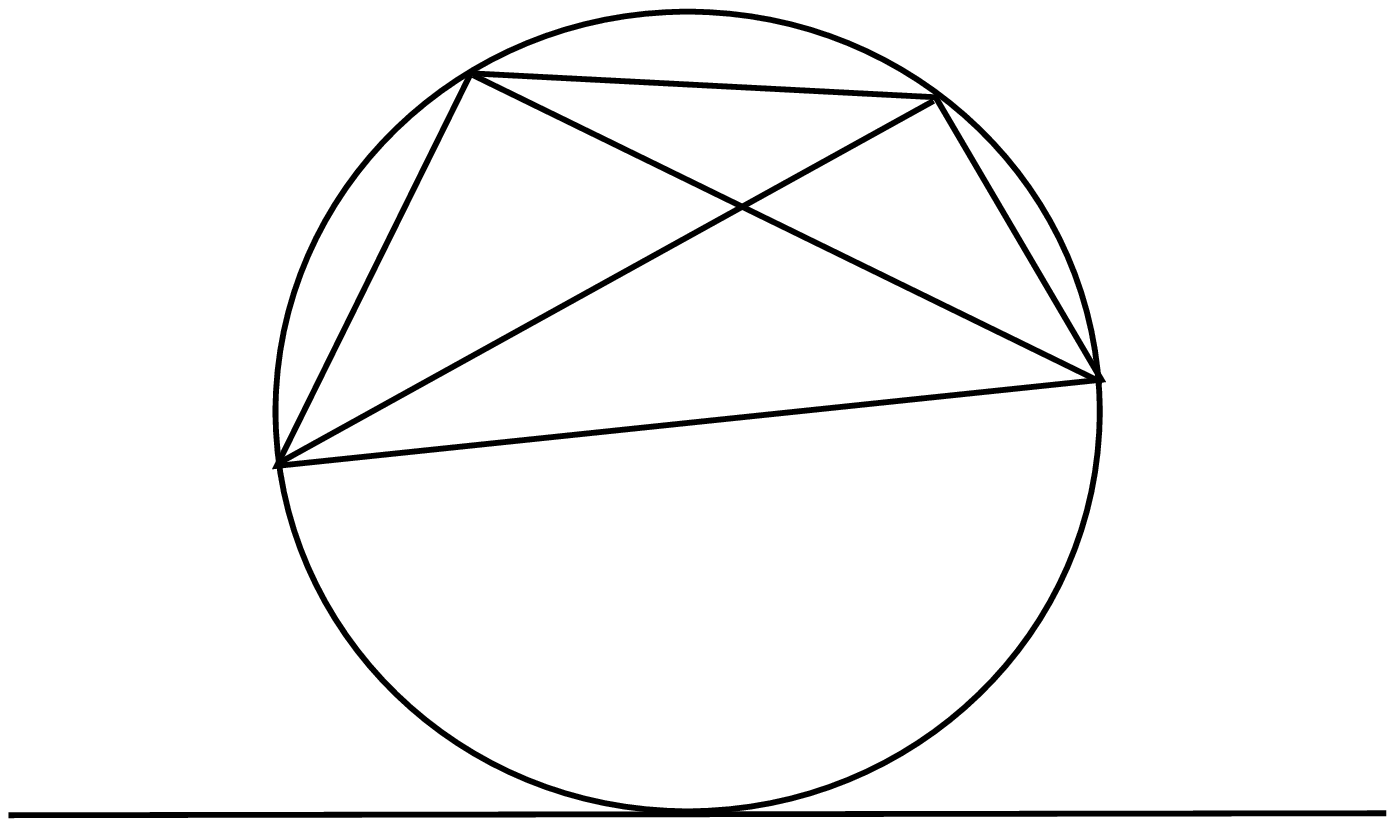}
\caption{} \label{fig:horo}
\end{figure}

\noindent Case 2. When the vertices lie on a horocycle, for
example as in Figure \ref{fig:horo}, we have
\begin{align*}
s(e)&=s(a)+s(b),\\
s(f)&=s(b)+s(c),\\
s(d)&=s(a)+s(b)+s(c).
\end{align*}
Then the equations in (iii) hold.

\noindent Case 3. When the vertices lie on a geodesic, without
loss of generality, we may assume

\begin{align*}
e&=a+b,\\
f&=b+c,\\
d&=a+b+c.
\end{align*}

Direct calculation shows that
$$s(a)s(c)+s(b)s(d)=s(a)s(c)+s(b)s(a+b+c)=s(a+b)s(c+b).$$
Similarly,
$$s(a)s(d)+s(b)s(c)=s(a+b)s(a+c),$$
$$s(a)s(b)+s(c)s(d)=s(c+a)s(c+b).$$
Therefore the right hand side of (\ref{fml:p1}) equals
$$s(a+b)s(c+b)\frac{s(a+b)s(a+c)}{s(c+a)s(c+b)}=s(a+b)^2=s(e)^2.$$

Similar argument proves the equation involving $s(f).$

\noindent Case 4. When the vertices lie on an equidistant curve
with distance $D$ to its geodesic axis, project the vertices to
the geodesic axis. The corresponding distance between those
projection of vertices are denoted by
$\overline{a},\overline{b},\overline{c},\overline{d},\overline{e},\overline{f}.$

By Case 3, we have
$$s(\overline{e})^2=(s(\overline{a})s(\overline{c})+s(\overline{b})s(\overline{d}))\frac{s(\overline{a})s(\overline{d})+s(\overline{b})s(\overline{c})}{s(\overline{a})s(\overline{b})+s(\overline{c})s(\overline{d})}.$$

Since
$$s(x)=s(\overline{x}) \cosh D $$ for $x=a,b,c,d,e,f,$ we have
$$s(e)^2=(s(a)s(c)+s(b)s(d))\frac{s(a)s(d)+s(b)s(c)}{s(a)s(b)+s(c)s(d)}.$$

\end{proof}

\subsection{Proof of Lemma \ref{thm:derivative}}

First, we verify that $$\frac{\partial A}{\partial
x}|_{A=B}=\frac{\partial B}{\partial x}.$$ The role of $x,y,z,w$
are the same with respect to $a.$ It is enough to verify the case
of variable $x$.

Now let $\alpha, \alpha', \beta, \beta'$ be the angles formed by
the pairs of edges $\{a,y\},\{a,x\},\{a,z\},\{a,w\}$ as Figure
\ref{fig:angle}.

\begin{figure}[ht!]
\labellist\small\hair 2pt \pinlabel $x$ at 159 537 \pinlabel $y$
at 299 553 \pinlabel $w$ at 183 419 \pinlabel $z$ at 322 434
\pinlabel $a$ at 206 484 \pinlabel $A$ at 237 532 \pinlabel
$\alpha$ at 343 496 \pinlabel $\beta$ at  342 477 \pinlabel
$\alpha'$ at 140 478 \pinlabel $\beta'$ at 150 463

\endlabellist
\centering
\includegraphics[scale=0.45]{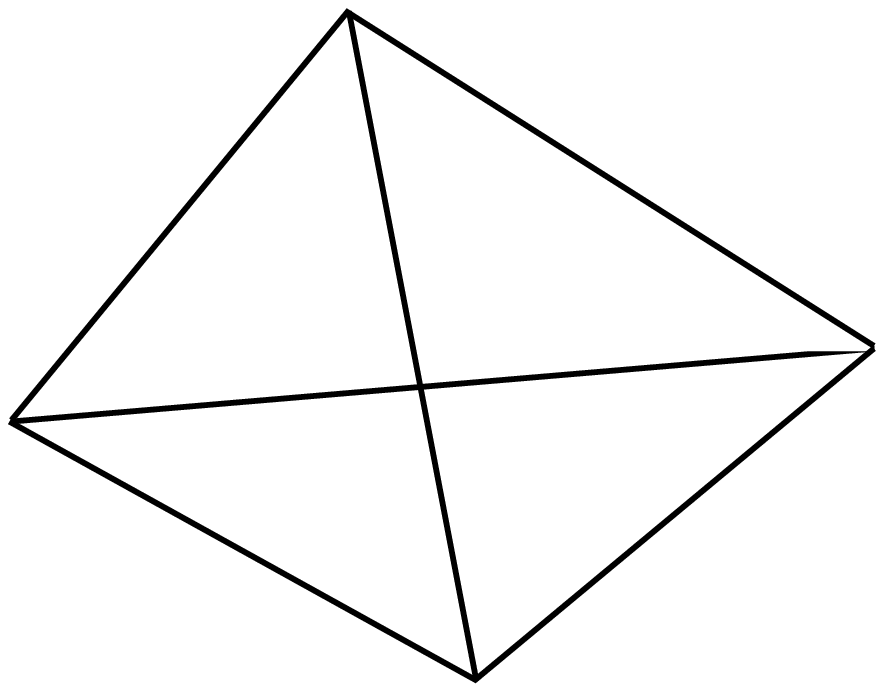}
\caption{} \label{fig:angle}
\end{figure}

In the triangle of lengths $y,z,A$, by the cosine rule,
$$\cosh A=\cosh y \cosh z- \sinh y \sinh z \cos(\alpha+\beta).$$
Taking derivative of both sides with respect to $x$, we have
$$\frac{\partial A}{\partial x}=\frac{\sinh y \sinh z \sin(\alpha+\beta)}{\sinh A} \cdot
\frac{\partial \alpha}{\partial x}.$$

In the triangle of lengths $x,y,a$, by the derivative of cosine
rule \cite{luo-rigidity}, we have
$$\frac{\partial \alpha}{\partial x}=\frac{\sinh x}{\sinh y \sinh a \sin \alpha}.$$

Therefore,
$$\frac{\partial A}{\partial x}=\frac{\sinh z }{ \sinh a}\cdot\frac{\sin(\alpha+\beta)}{\sinh A}\cdot \frac{\sinh x}{\sin\alpha}.$$

In the triangle of lengths $y, z, A$, Lemma \ref{thm:sin rule}
implies that
\begin{equation}\label{fml:ratio1}
\frac{\sinh A}{\sin(\alpha+\beta)}=2\zeta_1 \cosh \frac A2 \cosh
\frac y2 \cosh \frac z2.
\end{equation}

In the triangle of lengths $x,y,a$, Lemma \ref{thm:sin rule}
implies that
\begin{equation}\label{fml:ratio2}
\frac{\sinh x}{\sin \alpha}=2\zeta_2 \cosh \frac x2 \cosh \frac y2
\cosh \frac a2.
\end{equation}

Therefore,
\begin{align*}
\frac{\partial A}{\partial x}=&\frac{\sinh z }{ \sinh a}
\cdot\frac{2\zeta_2 \cosh \frac x2 \cosh \frac y2 \cosh \frac a2}{2\zeta_1 \cosh \frac A2 \cosh \frac y2 \cosh \frac z2}\\
=&\frac{\sinh \frac z2 \cosh \frac x2 \zeta_2}{\sinh \frac a2
\cosh \frac A2 \zeta_1}.
\end{align*}

When $A=B$, by Lemma \ref{thm:ptolemy}, the vertices of the
hyperbolic quadrilateral lie on a circle, a horocycle or an
equidistant curve. Thus $\zeta_1=\zeta_2$.

Therefore
$$\frac{\partial A}{\partial x}|_{A=B}=\frac{\sinh \frac z2 \cosh \frac x2}{\sinh \frac a2 \cosh \frac B2}=\frac{\partial B}{\partial x}.$$

Second, we verify that $$\frac{\partial A}{\partial
a}|_{A=B}=\frac{\partial B}{\partial a}.$$ In the triangle of
lengths $y,z,A$, by the cosine rule,
$$\cosh A=\cosh y \cosh z- \sinh y \sinh z \cos(\alpha+\beta).$$
Taking derivative of both sides with respect to $a$, we have
$$\frac{\partial A}{\partial a}=\frac{\sinh y \sinh z \sin(\alpha+\beta)}{\sinh A} \cdot
(\frac{\partial \alpha}{\partial a}+\frac{\partial \beta}{\partial
a}).$$

In the triangle of length $x,y,a$, by the derivative of cosine
rule \cite{luo-rigidity}, we have
$$\frac{\partial \alpha}{\partial a}=-\frac{\sinh x}{\sinh y \sinh a \sin \alpha}\cos \alpha'.$$

In the triangle of length $w,z,a$, by the derivative of cosine
rule \cite{luo-rigidity}, we have
$$\frac{\partial \beta}{\partial a}=-\frac{\sinh w}{\sinh z \sinh a \sin \beta}\cos \beta'.$$

Therefore
$$\frac{\partial A}{\partial a}=-\frac{\sin(\alpha+\beta)}{\sinh A \sinh a}(\frac{\sinh z\sinh x \cos \alpha'}{\sin \alpha}+
\frac{\sinh y\sinh w \cos \beta'}{\sin \beta}).$$

By the equations (\ref{fml:ratio1}) and (\ref{fml:ratio2}), we
have
$$\frac{\sin(\alpha+\beta)}{\sin \alpha}=\frac{\zeta_2 \cosh \frac a2 \sinh \frac A2}{\zeta_1 \cosh \frac z2 \sinh \frac x2}.$$

By the similar calculation, we have
$$\frac{\sin(\alpha+\beta)}{\sin \beta}=\frac{\zeta_3 \cosh \frac a2 \sinh \frac A2}{\zeta_1 \cosh \frac y2 \sinh \frac w2},$$
there $\zeta_3$ is the corresponding quantity of the triangle of
lengths $w,z,a$.

Therefore
$$\frac{\partial A}{\partial a}=-\frac 1{\cosh \frac A2 \sinh \frac a2}(\frac{\zeta_2}{\zeta_1}\sinh \frac z2 \cosh \frac x2 \cos \alpha'
+ \frac{\zeta_3}{\zeta_1}\sinh \frac y2 \cosh \frac w2 \cos
\beta').$$

When $A=B,$ by Lemma \ref{thm:ptolemy}, the vertices of the
hyperbolic quadrilateral lie on a circle, a horocycle or an
equidistant curve. Thus $\zeta_1=\zeta_2=\zeta_3.$

Therefore
$$\frac{\partial A}{\partial a}|_{A=B}=-\frac 1{\cosh \frac B2 \sinh \frac a2}(\sinh \frac z2 \cosh \frac x2 \cos \alpha'
+ \sinh \frac y2 \cosh \frac w2 \cos \beta').$$

On the other hand
$$\frac{\partial B}{\partial a}=-\frac{\sinh \frac B2 \cosh \frac a2}{ \cosh \frac B2 \sinh \frac a2}.$$

To prove $\frac{\partial A}{\partial a}|_{A=B}=\frac{\partial
B}{\partial a},$ it remains to show that
\begin{equation}\label{fml:eqa1}
\sinh \frac z2 \cosh \frac x2 \cos \alpha' + \sinh \frac y2 \cosh
\frac w2 \cos \beta'= \sinh \frac B2 \cosh \frac a2.
\end{equation}

In the triangle of length $x,y,a$, by the cosine rule,
$$\cos \alpha'=\frac{-\cosh y+ \cosh x \cosh a}{\sinh x \sinh a}.$$

In the triangle of length $w,z,a$, by the cosine rule,
$$\cos \beta'=\frac{-\cosh z+ \cosh w \cosh a}{\sinh w \sinh a}.$$

Therefore the equation (\ref{fml:eqa1}) is equivalent to
\begin{align}\label{fml:eqa2}
&\frac{\sinh \frac z2}{2\sinh \frac x2}(-\cosh y+ \cosh x \cosh a)+ \frac{\sinh \frac y2}{2\sinh \frac w2}(-\cosh z+ \cosh w \cosh a)\\
&=\sinh \frac B2 \cosh \frac a2 \sinh a.  \notag
\end{align}

Using the notation $s(t)=\sinh \frac t2$, we have $\cosh t=
2s(t)^2+1$. Therefore the equation (\ref{fml:eqa2}) is equivalent
to
\begin{align*}
&\frac{s(z)}{s(x)}(2s(a)^2s(x)^2+s(a)^2+s(x)^2-s(y)^2)\\
&+\frac{s(y)}{s(w)}(2s(a)^2s(w)^2+s(a)^2+s(w)^2-s(z)^2)\\
&=2s(B)s(a)(s(a)^2+1)\\
&=2(s(x)s(z)+s(y)s(w))(s(a)^2+1),
\end{align*}
the second equality is due to Ptolemy's formula.

After simplify we obtain
$$s(a)^2=(s(x)s(z)+s(y)s(w))\frac{s(x)s(w)+s(y)s(z)}{s(x)s(y)+s(z)s(w)}.$$
This is exactly the result of Lemma \ref{thm:ptolemy}.

\end{document}